\documentclass[12pt]{article}

\usepackage{amsfonts}
\usepackage{amssymb}
\usepackage{amsthm}
\usepackage{latexsym,amsmath}
\usepackage{graphicx}
\usepackage{xcolor}
\usepackage{mathrsfs}
\usepackage{tikz}\usetikzlibrary{calc}
\usepackage[margin=2cm]{geometry}
\usepackage{float}

\newtheorem{thm}{Theorem}[section]
\newtheorem{lem}[thm]{Lemma}

\newcommand{\F}{\mathscr{F}}
\newcommand{\G}{\mathscr{G}}
\newcommand{\fH}{\mathscr{H}}

\newcommand{\fP}{\mathscr{P}}
\newcommand{\K}{\mathcal{K}}

\newcommand{\n}{(n_1,\ldots,n_k)}
\theoremstyle{definition}
\newtheorem*{exa}{Example}

\title{Intersection theorems for families of matchings of complete $k$-partite $k$-graphs.}

\author{Adam Mammoliti \thanks{
School of Mathematics and Statistics
UNSW Sydney
NSW 2052, Australia
}\\
\texttt{adam.mammoliti@outlook.com.au}
}

\begin{document}
\date{}
\maketitle

\begin{abstract}
The celebrated {Erd\H{o}s-Ko-Rado} Theorem states that for $n \geq 2k$ a family $\F$ of $k$ subsets of $[n]$ for which each pair of members of $\F$ have a non-empty intersection has size at most $\binom{n-1}{k-1}$ and for $n >2k$ has exactly this size if and only if it is the family of all $k$-subsets of $[n]$ containing a fixed element $x\in [n]$.
Since its discovery, the {Erd\H{o}s-Ko-Rado} Theorem has been generalised extensively and many variants have been found for structures other than sets. 
One such variant is for permutations and so-called generalised permutations. These structures are equivalent to $r$-matchings of the complete bipartite graph $K_{n,m}$ with $r \leq \min\{n,m\}$ in a natural way. 

The culmination of results of several groups of authors constitute an {Erd\H{o}s-Ko-Rado} Theorem for families of generalised permutations and so for families of $r$-matchings of $K_{n,m}$ for all feasible values of $r,n$ and $m$. 
In this paper we generalise this by proving an {Erd\H{o}s-Ko-Rado} Theorem for families of $r$-matchings of complete $k$-partite $k$-graphs, which can be seen as a partial generalisation of the {Erd\H{o}s-Ko-Rado} Theorem itself.
We also prove similar results for $t$-intersecting families, and for families of matchings whose members have sizes from some set of integers $R$, rather than a single size $r$. 
\end{abstract}

\bigskip\noindent \textbf{Keywords:}
The Erd\H{o}s-Ko-Rado Theorem, permutations,  generalised permutations, 
intersecting, $t$-intersecting, complete bipartite graph,
complete $k$-partite $k$-graph, matching.

\noindent {\bf MSC subject classifications: 05D05, 05C65, 05A05 }

%

\section{Introduction}
For an integer $n$, let $[n]$ denote the set $\{ 1,2,\ldots,n\}$.
The power set  of a set $X$ is denoted by $2^{X}$
and the set of all subsets of $X$ of size $r$
is denoted by $\binom{X}{r}$.
A family of sets $\mathscr{F}$ is {\em intersecting}
if $A \cap B \neq \emptyset$ for all $A,B \in \mathscr{F}$ .
For a family of sets $\mathscr{G}$,
a subfamily of the form $\{ A \in \mathscr{G} \,:\, x \in A \}$
for some $x$ is called a {\em star} of $\G$.
The celebrated {Erd\H{o}s-Ko-Rado} Theorem is as follows.
\begin{thm}[The {Erd\H{o}s-Ko-Rado} Theorem~\cite{MR0140419}]\label{thm: EKR}
Let $n \geq 2r$ and $\mathscr{F} \subseteq \binom{[n]}{r}$ be an intersecting family.
Then
\[
|\mathscr{F}| \leq \binom{n-1}{r-1}\,.
\]
Furthermore, for $n>2r$ equality is attained if and only if $\mathscr{F}$ is a star of $\binom{[n]}{r}$.
\end{thm}

The {Erd\H{o}s-Ko-Rado} Theorem has been generalised extensively and many variants have been proven for structures other than sets.
One such variant and the focus of this paper is the set of {\em generalised permutations} $\fP_r(n,m)$,
defined for positive integers $r,n$ and $m$ with $r \leq \min\{n, m\}$ as follows:
\begin{align*}
\fP_r(n,m) = \{ \{(x_1, y_1), \ldots, (x_r, y_r)\} \;:\;& x_1,\ldots, x_r \textrm{ are distinct elements of } [n]\\
 & y_1,\ldots, y_r \textrm{ are distinct elements of } [m] \}\,.
\end{align*}
These indeed generalise permutations, because a permutation $\sigma$ of $[n]$ defines the unique generalised permutation 
\[
P_\sigma = \{(a,b)\in [n]\times[n]\;:\; \sigma(a)=b\}\,,
\]
of $\fP_n(n,n) $. 

There has been extensive work on {Erd\H{o}s-Ko-Rado} type theorems for permutations and generalised permutations;
we provide a brief outline of some of those results relevant to us here.
Deza and Frankl~\cite{MR0439648}
proved that a family of intersecting permutations,
i.e., an intersecting subfamily of $\fP_n(n,n)$,
has size at most $(n-1)!$.
Cameron and Ku~\cite{MR2009400} and Larose and Malvenuto~\cite{MR2061391}
independently showed that an intersecting family of permutations
has size $(n-1)!$ if and only if the family is a star of $\fP_n(n,n)$.
Larose and Malvenuto (see \cite[Theorem 5.1]{MR2061391})
also proved that intersecting subfamilies of $\fP_r(r,n)$
have size at most $(n-1)_{r-1}$, where
for non-negative integers $a$ and $b$, $(a)_b$ denotes $\frac{a!}{(a-b)!}$.
Ku and Leader~\cite{MR2202076}
showed that $|\mathscr{F}| \leq \frac{(n-1)_{r-1} (n-1)_{r-1}}{(r-1)!}$
for any intersecting family $\mathscr{F} \subseteq \fP_r(n,n)$ and proved that only stars have the maximum size, except for a small number of cases.
These remaining cases were subsequently shown by Li and Wang~\cite{MR2285800}.
Borg and Meagher~\cite{MR3301134} proved that
$|\mathscr{F}| \leq \frac{(n-1)_{r-1} (m-1)_{r-1}}{(r-1)!}$ for
intersecting families $\mathscr{F} \subseteq \fP_r(n,m)$ for $r < \min\{n,m\}$
and, furthermore, that equality is attained if and only if $\mathscr{F}$ is a star of $\fP_r(n,m)$.

We encapsulate all of the results above in the following theorem.

\begin{thm}\label{thm: gen perm}
Let $\mathscr{F} \subseteq \fP_r(n,m)$ be an intersecting family, where $r,n$ and $m$ are positive integers such that $r \leq \min\{n,m\}$.
Then
\[
|\mathscr{F}| \leq \frac{(n-1)_{r-1} (m-1)_{r-1}}{(r-1)!}\,.
\]
Furthermore, equality is attained if and only if $\mathscr{F}$ is a star of $\fP_r(n,m)$.
\end{thm}
\noindent
In this paper, we will consider a generalisation of generalised permutations.
We first reinterpret generalised permutations in terms of hypergraphs.
A {\em hypergraph} $(V,E)$ is a pair consisting of a set~$V$
whose elements are called {\em vertices} and a set $E\subseteq 2^V$ whose elements are called {\em edges}.
A {\em $k$-graph} is a hypergraph for which every edge has size $k$, in particular, a $2$-graph is just a graph.
The {\em complete $k$-partite $k$-graph} with parts of sizes $n_1,\ldots ,n_k$,
denoted $\K\n$,
is the $k$-graph whose vertex set is the disjoint union of the sets $[n_1], \ldots, [n_k]$
 and whose edge set is exactly every $k$-set which contains exactly one vertex
in $[n_i]$ for all $i$.
In particular, $\K(n_1)$ is the set of singletons of $[n_1]$
and $\K(n_1,n_2)$ is the complete bipartite graph with parts $[n_1]$ and~$[n_2]$;
we (partially) conform to convention by denoting the latter graph as $K(n_1,n_2)$.

An {\em $r$-matching} of a hypergraph is a set of $r$ vertex disjoint edges of that hypergraph.
We let $\K_r\n $ and $K_r(n,m) $ denote
the set of $r$-matchings of $\K\n$ and $K(n,m) $, respectively.
By associating each element 
$(x,y) \in [n]\times [m]$ to the edge of $K(n,m)$ incident to $x$ and $y$,
there is a natural equivalence between generalised permutations in $\fP_r(n,m)$ and $r$-matchings in $K_r(n,m)$. Therefore, the set of $r$-matchings $\K_r\n $ can be seen as a generalisation of the set of generalised permutations $\fP_r(n,m)$. 
Moreover for a family $\F \subseteq K_r(n,m)$ two members of its equivalent counterpart $\G\subseteq \fP_r(n,m)$ intersect if and only if the corresponding members of $\F$ share a common edge.
Thus, we say a family $\F \subseteq \K_r\n$ is {\em intersecting} if every pair of members of $\F$ share a common edge.

Considering intersecting families of graphs and hypergraphs is not new. 
We briefly outline some of the previous work related to intersecting families of graphs and hypergraphs.
Meagher and Moura \cite{MeMo} proved an {Erd\H{o}s-Ko-Rado} Theorem for intersecting families of perfect matchings, i.e.~$(n/k)$-matchings,
of the complete $k$-graph on $n$ vertices $\K^k_n$, with $k|n$ and in particular an {Erd\H{o}s-Ko-Rado} Theorem for intersecting families of perfect matchings of the complete graph $K_n$ with $n$ even. 
Let $\fH$ be a family of graphs on a common vertex set $V$. 
We say a family of labelled graphs $\F$ on the vertex set $V$ is $\fH$-intersecting if the intersection of each pair of members of $\F$ contains a member of $\fH$ as a subgraph. 
For a graph $H$, a family of graphs is $H$-intersecting if it is $\fH$-intersecting, 
where $\fH$ is the family of all labelled graphs isomorphic to $H$.
Simonovits and Sós~\cite{SiSo76,SiSo78} proved intersection theorems for families of $\fH$-intersecting graphs, for various families $\fH$.
Ellis, Filmus and Friedgut~\cite{EFF} prove a long standing conjecture of Simonovits and Sós which asserted that a maximum sized $K_3$-intersecting family of graphs is a family consisting of all the graphs containing a fixed $K_3$.
Berger and Zhao~\cite{BeZh} prove an analogous result for $K_4$-intersecting families as well as cross intersecting variants and stability results.

The main result of this paper is
the following generalisation of Theorem~\ref{thm: gen perm} to subfamilies of  $\K_r\n$.

\begin{thm}\label{thm: Gen int perms} 
Let $\mathscr{F} \subseteq  \K_r\n$ be an intersecting family, where
$k\geq 2$, $r\geq 1$ and $n_1,\ldots, n_k$ are integers that are at least $r$.
Then
\[
  |\mathscr{F}| \leq \frac{(n_1-1)_{r-1}  \cdots (n_k-1)_{r-1}}{(r-1)!} \,.
\]
Furthermore, equality is attained if and only if $\mathscr{F}$ is a star of  
$\K_r\n$.
\end{thm}
\noindent
In fact, as  $\K_r(n_1)$ is equivalent to $\binom{[n_1]}{r}$,
Theorem~\ref{thm: Gen int perms} can be seen as a generalisation of the {Erd\H{o}s-Ko-Rado} Theorem for $n_1 > 2r$. Theorem~\ref{thm: Gen int perms} will be proven using a relatively simple inductive argument on $k$, where Theorem~\ref{thm: gen perm} forms the base case.
We also prove {Erd\H{o}s-Ko-Rado} type theorems for families $\F$ of matchings; whose members may have any size from a set of integers $R$ rather than just a single size $r$ (Theorem~\ref{thm: gen int hyperperm non-uniform}), for which each pair of members of $\F$ share $t$ common edges (Theorems~\ref{thm: Gen t-int perms hyper} and~\ref{thm:t-intHyperPerm}) or both (Theorem~\ref{thm:t-int gen perms hyper non-uniform}).

The paper is organised as follows.
In Section~\ref{sec:definitions} we present definitions and preliminary results.
The proof of Theorem~\ref{thm: Gen int perms} is given in Section~\ref{sec:MainResults}.
This is followed by the proof of Theorem~\ref{thm: gen int hyperperm non-uniform}. 
Section~\ref{sec:MainResults} ends with the proofs of Theorems~\ref{thm: Gen t-int perms hyper}--\ref{thm:t-int gen perms hyper non-uniform}.
The paper concludes with Section~\ref{sec:conclusion} in which we briefly discuss previous conjectures and open questions as well as new ones that naturally arise from the work in this paper.

\section{Definitions and preliminary results}
\label{sec:definitions}

For what follows, let $k,t, r, n_1, \ldots, n_k$ be integers such that $k \geq 2$ and $1 \leq t \leq r \leq n_1 \leq \cdots \leq n_k$.
Let $\mathcal{K}(n_1,\ldots, n_k)$ denote the complete $k$-partite $k$-graph
with parts $[n_1], \ldots, [n_k]$, as defined in the Introduction.
For the purposes of this paper, we take the edge set of $\K(n_1,\ldots, n_k)$ to be the set of  $k$-tuples in $[n_1] \times \cdots \times[n_k]$, where for all $y \in [n_i]$ with $i\in[k]$, $y$ is incident to the edge $(x_1,\ldots,x_k) \in [n_1] \times \cdots \times[n_k]$ if and only if $x_i=y$.
Let $\K_r\n$ be the set of $r$-matchings of $\mathcal{K}(n_1,\ldots, n_k)$.
That is
\[
\K_r(n_1,\ldots, n_k) = \{ \{(x_{i,1}, x_{i,2},\ldots, x_{i,k}) \;:\; i \in [r] \} \;:\; x_{1,j},\ldots, x_{r,j} \in [n_j] 
\textrm{ are distinct for all }j \} \,.
\]
We abuse notation slightly by identifying $\K_1\n$ with the edge set of $\K\n$.

Let $M \in \K_r\n$ and $\F \subseteq \K_r\n$.
For an edge $e =(x_1,\ldots ,x_k) \in \K_1\n$ we define 
$R_j(e) = (x_1,\ldots,x_{j-1},x_{j+1},\ldots , x_k)$.
We further define
\[
R_j(M) = \{R_j(e)\;:\; e \in M\} \quad \text{ and } \quad R_j(\F) = \{R_{j}(M)\;:\; M \in \F\}\,.
\]
For $i \neq j$ and $e \in \K_1\n$, we define $S_j^i(e) = (x_i,x_j) \in K_1(n_i,n_j)$ where $x_i$ and $x_j$ are the vertices in $[n_i]$ and $[n_j]$ incident to $e$, respectively.
We further define
\[
S_j^i(M) = \{S_j^i(e)\;:\; e \in M\} \quad \text{ and } \quad S_j^i(\F) = \{S_j^i(M)\;:\; M \in \F\}\,.
\]


For a graph or a set of edges $M$, we let $V(M)$ denote the vertices of $M$.
Moreover, if $M$ is in fact a set of edges of $\K\n$,
then we let $V_{i}(M)$ denote the set of vertices in $[n_i]$ incident to an edge in $M$.
Notice that, for each $M \in \K_r\n$, the pair $(R_j(M),S_j^i(M))$ is unique for each $i \neq j $.
Conversely, for $X \in R_j(\K_r\n)$ and $Y \in K_r(n_i,n_j)$, if $V_i(X) = V_i(Y)$, then there is a unique $r$-matching $M$ in $\K_r\n$ such that $R_j(M)=X$ and $S_j^i(M)=Y$, namely
\[
M = \left\lbrace e \cup \{x_j\} \;:\; e \in X \text{ and } (x_i,x_j) \in Y \text{ for the $x_i \in e \cap [n_i]$}\right\rbrace.
\]
Therefore, for $X \in R_j(\K_r\n)$ and $Y \in K_r(n_i,n_j)$ such that $V_i(X) = V_i(Y)$, we define $X \ltimes_i Y$ to be the unique matching $M\in \K_r\n$ such that $R_j(M)=X$ and $S_j^i(M)=Y$
and say that $X$ and $Y$ are {\em compatible}.
\begin{figure}[hb]
\begin{center}
\begin{tikzpicture}[scale=1.125] 
  \pgfmathsetlengthmacro\scface{1.5cm}
  \pgfmathsetlengthmacro\scfacv{2cm}
  \pgfmathsetlengthmacro\n{.25}

\draw[color= black]{
(1.5*\scface,-0.75*\scface)node[color=black]{\text{}$M$}
};
\draw[color= black]{
(1.5*\scface,3.75*\scface)node[color=black]{\text{A matching of }$\mathcal{K}(4,4,4,4)$}
};
\begin{scope}
\filldraw[color=black,fill=black!50, line width = \scfacv*0.045,opacity=0.8] 
($(0,2*\scface)+(-135:\n*\scface)$) arc(225:45:\n*\scface) 
-- ($(1*\scface,1*\scface)+(45:\n*\scface)+(135:-\n*\scface+\n*\scface*sqrt 2)$)
-- ($(2*\scface,1*\scface)+(135:\scface*\n)+(45:-\n*\scface+\n*\scface*sqrt 2)$)
-- ($(3*\scface,2*\scface)+(135:\n*\scface)$) arc(135:-45:\n*\scface) 
-- ($(2*\scface,1*\scface)+(-45:\scface*\n)$)arc(-45:-90:\n*\scface) 
--($(1*\scface,1*\scface)+(-90:\n*\scface )$)arc(-90:-135:\n*\scface) 
--cycle;
\filldraw[color=black,fill=black!50, line width = \scfacv*0.045,opacity=0.8] 
($(0,0*\scface)+(-90:\n*\scface)$) arc(270:90:\n*\scface) 
-- ($(3*\scface,0*\scface)+(90:\n*\scface)$)arc(90:-90:\n*\scface) 
--cycle;
\filldraw[color=black,fill=black!50, line width = \scfacv*0.045,opacity=0.8] 
($(0,1*\scface)+((-45: \scface*\n)$) arc(-45:-45-180:\scface*\n)
-- ($(2*\scface,3*\scface)+(-45-180: \n*\scface)$)arc(135:90:\n*\scface)
-- ($(3*\scface,3*\scface)+(90: \n*\scface)$)arc(90:-90:\n*\scface)
-- ($(2*\scface,3*\scface)+(-45: \n*\scface)+(-135:-\n*\scface+\n*\scface*sqrt 2)$)
--cycle; 
\filldraw[color=black,fill=black!50, line width =  \scfacv*0.045,inner sep=\scfacv*0.03, minimum width=\scfacv*0.08,opacity=0.8]  
($(0,3*\scface)+(-90:\n*\scface)$) arc(-90:-270:\n*\scface) 
-- ($(1*\scface,3*\scface)+(90:\n*\scface)$)arc(90:45:\n*\scface) 
-- ($(3*\scface,1*\scface)+(45:\n*\scface)$)arc(45:-135:\n*\scface) 
-- ($(1*\scface,3*\scface)+(-135:\n*\scface) +(-45:-\n*\scface+\n*\scface*sqrt 2)$)
--cycle;  
\end{scope}
\begin{scope}
 \foreach \y in {0,1,2,3}
{
\draw[line width = \scfacv*0.015, color = black!50]{
(0,\y*\scface) node[circle, draw, fill=black!10,inner sep=\scfacv*0.03]{
}
(1*\scface,\y*\scface) node[circle, draw, fill=black!10,inner sep=\scfacv*0.03]{
}
(2*\scface,\y*\scface) node[circle, draw, fill=black!10,inner sep=\scfacv*0.03]{
}
(3*\scface,\y*\scface) node[circle, draw, fill=black!10,inner sep=\scfacv*0.03]{
}
};
}
\end{scope}
\end{tikzpicture}
\quad
\begin{tikzpicture}[scale=1.125] 
  \pgfmathsetlengthmacro\scface{1.5cm}
  \pgfmathsetlengthmacro\scfacv{2cm}
  \pgfmathsetlengthmacro\n{.25}
\begin{scope}[blend group = screen]

\draw[color= black]{
(1.5*\scface,3.75*\scface)node[color=black]{\text{A pair of matchings of $\mathcal{K}(4,4,4)$ and $ K(4,4)$}}
};
\draw[color= black]{
(1.5*\scface,-0.75*\scface)node[color=black]{\text{$R_4(M)$ and $S_4^3(M)$}}
};

\filldraw[color=blue,fill=blue!50, line width = \scfacv*0.045,opacity=0.8] 
($(0,2*\scface)+(-135:\n*\scface)$) arc(225:45:\n*\scface) 
-- ($(1*\scface,1*\scface)+(45:\n*\scface)+(135:-\n*\scface+\n*\scface*sqrt 2)$)
-- ($(2*\scface,1*\scface)+(90:\scface*\n)$)arc(90:-90:\n*\scface) 
--($(1*\scface,1*\scface)+(-90:\n*\scface )$)arc(-90:-135:\n*\scface) 
--cycle;

\filldraw[color=red,fill=red!50, line width = \scfacv*0.045,opacity=0.8] 
($(2*\scface,1*\scface)+(-45:\scface*\n)$)arc(-45:-45-180:\n*\scface)
-- ($(3*\scface,2*\scface)+(135:\n*\scface)$) arc(135:-45:\n*\scface) 
--cycle;

\filldraw[color=blue,fill=blue!50, line width = \scfacv*0.045,opacity=0.8] 
($(0,0*\scface)+(-90:\n*\scface)$) arc(270:90:\n*\scface) 
-- ($(2*\scface,0*\scface)+(90:\n*\scface)$)arc(90:-90:\n*\scface) 
--cycle;

\filldraw[color=red,fill=red!50, line width = \scfacv*0.045,opacity=0.8] 
($(2*\scface,0*\scface)+(-90:\n*\scface)$) arc(270:90:\n*\scface) 
-- ($(3*\scface,0*\scface)+(90:\n*\scface)$)arc(90:-90:\n*\scface) 
--cycle;

\filldraw[color=blue,fill=blue!50, line width = \scfacv*0.045,opacity=0.8] 
($(0,1*\scface)+(-45: \scface*\n)$) arc(-45:-45-180:\scface*\n)
-- ($(2*\scface,3*\scface)+(135: \scface*\n)$)arc(135:-45:\n*\scface)
--cycle;

\filldraw[color=red,fill=red!50, line width = \scfacv*0.045,opacity=0.8] 
($(2*\scface,3*\scface)+(-90: \n*\scface)$)arc(-90:-270:\n*\scface)
--($(3*\scface,3*\scface)+(90: \n*\scface)$)arc(90:-90:\n*\scface)
--cycle;

\filldraw[color=blue,fill=blue!50, line width =  \scfacv*0.045,inner sep=\scfacv*0.03, minimum width=\scfacv*0.08,opacity=0.8]  
($(0,3*\scface)+(-90:\n*\scface)$) arc(-90:-270:\n*\scface) 
-- ($(1*\scface,3*\scface)+(90:\n*\scface)$)arc(90:45:\n*\scface) 
-- ($(2*\scface,2*\scface)+(45:\n*\scface)$)arc(45:-135:\n*\scface) 
-- ($(1*\scface,3*\scface)+(-135:\n*\scface) +(-45:-\n*\scface+\n*\scface*sqrt 2)$)
--cycle;

\filldraw[color=red,fill=red!50, line width =  \scfacv*0.045,inner sep=\scfacv*0.03, minimum width=\scfacv*0.08,opacity=0.8]  
($(2*\scface,2*\scface)+(-135:\n*\scface)$)arc(-135:-315:\n*\scface) 
--($(3*\scface,1*\scface)+(45:\n*\scface)$)arc(45:-135:\n*\scface) 
--cycle;  
\end{scope}
\begin{scope}
 \foreach \y in {0,1,2,3}
{
\draw[line width = \scfacv*0.015, color = black!50]{
(0,\y*\scface) node[circle, draw, fill=black!10,inner sep=\scfacv*0.03]{
}
(1*\scface,\y*\scface) node[circle, draw, fill=black!10,inner sep=\scfacv*0.03]{
}
(2*\scface,\y*\scface) node[circle, draw, fill=black!10,inner sep=\scfacv*0.03]{
}
(3*\scface,\y*\scface) node[circle, draw, fill=black!10,inner sep=\scfacv*0.03]{
}
};
}
\end{scope}
\end{tikzpicture}
\end{center}\caption{A matching $M$ (in black) and its associated compatible matchings $R_4(M)$ (in blue) and $ S_4^3(M)$ (in red).}\label{f:ex}

\end{figure}
Figure~\ref{f:ex} depicts an example of a matching and one of its associated compatible pairs.
For $X \in R_{j}(\K_r\n)$ and $i \neq j$ let
\[
N_X^i(\mathscr{F}) =
\left\lbrace Y\in S_j^i(\F)
\;:\;  X \ltimes_i Y \in \F \right\rbrace . 
\]

We say that a family $\F \subseteq \K_r\n$ is $t$-intersecting if any two members of $\F$ share at least $t$ edges.
In particular, a $1$-intersecting family is just an intersecting family as defined in the Introduction; we continue to use the latter terminology.
For a family $\G$ and a $t$-set $C$, the {\em $t$-star of $\G$ with centre $C$} is the family 
$\{G \in \G : C \subseteq G\}$. When the context is clear we will simply refer to such a family as a $t$-star. 
We also refer to a $1$-star simply as a star, which agrees with the definition of a star given in the Introduction.
We now note several simple properties of the definitions given thus far. 
Below and henceforth, we let $K(A,B)$ denote the complete bipartite graph with parts $A \subseteq [n_i]$ and $B \subseteq [n_j]$, which is considered to be an induced subgraph of $K(n_i,n_j)$
and let $K_r(A,B)$ be the set of $r$-matchings of $K(A,B)$.

\begin{lem}\label{lem:properties}
Let $ \F \subseteq \K_r\n$, $i,j \in [k]$ be distinct integers and $X \in R_{j}( \F)$.
Then the following hold:
\begin{itemize}
\item[{\rm(i)}]   if $\mathscr{F}$ is $t$-intersecting,
                  then $R_j(\mathscr{F})$, $S_j^i(\F)$ and $N_X^i(\mathscr{F}) $ are $t$-intersecting;
\item[{\rm(ii)}]  $N^i_X(\F) \subseteq K_r\big(V_i(X),[n_j]\big)$;
\item[{\rm(iii)}] if $\F$ is the $t$-star of $\K_r\n$ with centre $C$, then $R_j(\F)$, $S_j^i(\F)$ and $N_X^i(\F)$ are the $t$-stars
of $R_j(\K_r\n)$, $K_r(n_i,n_j)$ and $K_r\big(V_i(X),[n_j]\big)$ with centres $R_j(C)$, $S_j^i(C)$ and $S_j^i(C)$, respectively;
\item[{\rm(iv)}]  $|\mathscr{F}| = \sum_{X \in R_j(\mathscr{F})} |N_X^i(\mathscr{F})|$.
\end{itemize}
\end{lem}

\begin{proof}
If $L,M \in \F$ and $ L \cap M=T$, then clearly $R_j(L)\cap R_j(M) \supseteq R_j(T)$ 
and $S_j^i(L)\cap S_j^i(M)\supseteq S_j^i(T)$. Hence, if $\F$ is $t$-intersecting then so are $R_j(\F)$, $S_j^i(\F)$ 
and $N_X^i(\F)\subseteq S_j^i(\F)$ for all $X \in R_j(\F)$.
This establishes (i).
By definition, every $Y \in N_X^i(\F)$ is compatible with $X$ and in particular $V_i(Y)=V_i(X)$.
It follows that $Y \in K_r(V_i(X),[n_j])$ for all $Y \in N_X^i(\F)$, proving (ii).

To prove (iii), let $\F$ be the $t$-star of $\K_r\n$ with centre $C$. 
We prove that $R_j(\F)$, $S_j^i(\F)$ and $N_X^i(\F)$ for each $X \in R_j(\F)$ are the appropriate $t$-stars. 
Notice that for every $M\in\F$, $R_j(M) \supseteq R_j(C)$ and $S_j^i(M) \supseteq S_j^i(C)$,
since $C \subseteq M$. 
So $R_j(\F)$ is a subset of the $t$-star of $R_j(K_r\n)$ with centre $R_j(C)$, 
$S_j^i(\F)$ is a subset of the $t$-star of $K_r(n_i,n_j)$ with centre $S_j^i(C)$,  and 
$N_X^i(\F)$ is a subset of the $t$-star of $K_r\big(V_i(X),[n_j]\big)$ with centre $S_j^i(C)$.
We complete the proof of (iii) by showing $R_j(\F)$,
$S_j^i(\F)$ and $N_X^i(\F)$ contain the appropriate $t$-stars.
Let $R_j(C) \subseteq X' \in R_j(\K\n)$ and $S_j^i(C) \subseteq Y' \in K_r\big(V_i(X'),[n_j]\big)$.
Then $M = X'\ltimes_i Y'$ is an $r$-matching of $\K\n$ containing $C= R_j(C)\ltimes_i S_j^i(C)$ and so $M \in \F$, since $\F$ is the $t$-star of $\K_r\n$ with centre $C$.
It follows that $X' \in R_j(\F)$ and $Y' \in N_{X'}^i(\F) \subseteq S_j^i(\F)$ and hence 
(iii) holds, since for any $Y'$ such that $S_j^i(C) \subseteq Y' \in K_r(n_i,n_j)$, $Y'$ is in $K_r(V_i(X'),[n_j])$
for some $R_j(C) \subseteq X' \in R_j(\K\n)$.

Finally we prove (iv). Let $i,j \in [k]$ be distinct integers. Recall that $(R_j(M),S_j^i(M))$ is unique for each $M \in \K_r\n$.
Therefore,
\begin{align*}
|\mathscr{F}|
&= |\{(R_j(M),S_j^i(M))\;:\; M \in \F \}| \\
&=\sum_{X \in R_j(\mathscr{F})} |\{(X,Y) \;:\; Y \in S_j^i(\F) \text{ and }X\ltimes_i Y\in\F\}|
= \sum_{X \in R_j(\F)}|N^i_X(\F)|\,,
\end{align*}
establishing statement (iv), where the second equality holds as the summation is a simple restatement of the expression preceding it
and the final equality holds by the definition of $N^i_X(\F)$.
\end{proof}

The machinery required to prove the bound of Theorem~\ref{thm: Gen int perms} has been given in Lemma~\ref{lem:properties}.
Establishing the uniqueness of families of maximum size requires more, most of which is encapsulated in Lemma~\ref{lem:whenSubfamststar}. 
Before proving the lemma, we require (a special case of) the following result.
In the lemma below, we let $\K(A_1,\ldots ,A_k)$ denote the complete $k$-partite $k$-graph whose $i$-th part is $A_i \subseteq [n_i]$ for each $i \in [k]$, which is considered to be an induced subhypergraph of $\K\n$
and we let $\K_r(A_1,\ldots ,A_k)$ be the set of $r$-matchings of $\K(A_1,\ldots ,A_k)$.
\begin{lem}\label{l:distMsInStars}
Let $r,r'\geq t$ be positive integers.
Let $A_i \subseteq [n_i]$ and $B_i \subseteq [n_i]$ for all $i\in [k]$ where for some $j\in[k]$, $|A_j|\geq t+2$ and $|B_j|\geq t+2$. If $\G$ and $\G'$ are the non-empty $t$-stars of $\K_r(A_1,\ldots ,A_k)$ and $\K_{r'}(B_1,\ldots ,B_k)$ with centres $C$ and $C'$, respectively, then there exists $M \in \G$ and $M'\in \G'$ such that $M \cap M' = C \cap C'$.
\end{lem}
\begin{proof}
Suppose for a contradiction that for every  $C \subseteq M \in \G$ and $C' \subseteq M' \in \G'$, 
$M \cap M' \supsetneq C \cap C'$. Let $M$ and $M'$ be two such matchings whose intersection is minimal.
By assumption, there is some $e=(x_1,\ldots, x_k)\in M \cap M'$ that is not in $C \cap C'$.
Without loss of generality, suppose that $e\notin C$.
Note that $x_i \notin V_i(C)$ for all $i \in [k]$, since $C \cup \{e\}$ is contained in the matching $M$.
As $|A_j|\geq t+2$, there exists an element $y_j \in A_j-V_j(C)$ other than $x_j$.
Let $e' = (e-\{x_j\})\cup \{y_j\}$. 
If $y_j\in V_j(M)$, let
$f = (y_1,\ldots, y_k)$ be the edge in $M$ incident to $y_j$
and let $f' = (f-\{y_j\})\cup\{x_j\}$.
Note that, when $f$ exists, $y_i \neq x_i$ and $y_i \notin V_i(C)$ for all $i\in [k]$.
Let
\[
L =
\begin{cases}
  \big(M-\{e\}\big)\cup \{e'\} & \text{if $y_j \notin V_j(M)$} \\
  \big(M- \{e,f\}\big)\cup \{e',f' \}&\text{otherwise. }
\end{cases}
\]
Clearly $L$ is an $r$-matching of $\K(A_1,\ldots, A_k)$ containing $C$ and so $L \in \G$. Yet, as $M$ and $M'$ are matchings containing $e$, $ L \cap M' \subsetneq M \cap M'$, contradicting the minimality of $M \cap M'$.
\end{proof}
In particular, the lemma shows that two $t$-stars with different centres can not both be contained in a $t$-intersecting family, in general. This observation will prove useful in the proof of Lemma~\ref{lem:whenSubfamststar} soon to follow and in Section~\ref{sec:MainResults}. 
However, the lemma fails if $|A_i| \leq t+1$ for all $i\in [k]$, as the following example shows.
Let $A_1,\ldots, A_k$ and $B_1,\ldots, B_k$ be sets  such that $A_i \in\binom{[n_i]}{t+1}$, $B_i \subseteq [n_i]$ has size $r'\geq t+1$ and $A_i \cap B_i \neq \emptyset$ for all $i\in[k]$. 
Let $e=(x_1,\ldots,x_k) \in \K_1\n$ be an edge such that $x_i \in A_i \cap B_i$ for all $i \in [k]$. 
Let $C \in \K_{t}(A_1,\ldots ,A_k)$ and $C' \in \K_{t}(B_1,\ldots ,B_k)$ so that $x_i \notin V_i(C)$
for all $i\in [k]$ and $e \in C'$. Then $A_i-V_i(C) =\{x_i\}$, since $|A_i|= t+1 = |V_i(C)|+1$ and so $\G$, the $t$-star of $\K_{t+1}(A_1,\ldots, A_k)$ with centre $C$, contains the single $(t+1)$-matching $M=C \cup \{e\}$. 
On the other hand, every matching $M'$ in $\G'$, the $t$-star of $\K_{r'}(B_1,\ldots, B_k)$ with centre $C'$, contains $C'$ and so
\[
M \cap M'=(C\cap M')\cup(\{e\}\cap M') \supseteq (C\cap C')\cup\{e\}\,.
\]
Moreover, if $n_i =t+1$ for all $i$, let $A_i = [t+1]=B_i$ for $i \in [k]$,
$C$ and $C'$ to be distinct $t$-subsets of some $(t+1)$-matching $T$ of $\K\n$ (and $e\in C'-C$)
and $\G$ and $\G'$ be the $t$-stars of $\K_{r}\n$ with centres $C$ and $C'$, respectively.
Then $\G=\{T\}=\G'$, but both are $t$-stars with different centres, namely $C$ and $C'$, respectively.
This demonstrates that two $t$-stars of $\K_{r}\n$ with different centres can be the same family.
However, this can only occur if $r=n_1=\cdots =n_k=t+1$, otherwise any two stars with different centres are distinct.



We can now prove Lemma~\ref{lem:whenSubfamststar}. Recall that $n_1 \leq \cdots \leq n_k$.

\begin{lem}\label{lem:whenSubfamststar}
Let $\F \subseteq \K_r\n$ be a $t$-intersecting family where $k\geq 3$ and $n_k \geq t+2$.
If $R_k(\F)$ is a $t$-star of $R_k(\K_r\n)$ and $N_X^1(\F)$ are $t$-stars of $K_r(V_1(X),[n_k])$ for all $X \in R_k(\F)$, then $\F$ is a $t$-star of $\K_r\n$.
\end{lem}
\begin{proof}
The result is clear if $r=t$, as then $\F$, $R_k(\F)$ and $N^1_X(\F)$ each have one member. 
So we can assume that $r \geq t+1$. 
By Lemma~\ref{lem:properties}(i), $S_k^1(\F)$ is $t$-intersecting and, as $n_k\geq t+2$, 
each $t$-star $N_X^1(\F) \subseteq S_k^1(\F)$ must have the same centre for all $X\in R_k(\F)$,
by Lemma~\ref{l:distMsInStars}.
Therefore, we let $C_S$ denote the common centre of all $N_X^1(\F)$ with $X\in R_k(\F)$.
If $n_i = t+1$ for all $i\neq k$, then $R_k(\F)$ has a single member $T$ and necessarily $V_1(T)=[n_1]$. So $R_k(\F)$ is the $t$-star with centre $C_R$, where $C_R$ is the $t$-subset of $T$ such that $V_1(C_R)=V_1(C_S)$.
Otherwise, we let $C_R$ be the unique centre of $R_k(\F)$.

First we show that $V_1(C_R)=V_1(C_S)$, which is immediate if $n_i = t+1$ for all $i \neq k$. 
So we can assume that $n_\ell \geq t+2$ for some $\ell< k$.
Using Lemma~\ref{l:distMsInStars} where $\G=R_k(\F)=\G'$,
there exists distinct $X,X' \in R_k(\F)$ such that $X \cap X' = C_R$.
Similarly as $n_k \geq t+2$, there exists distinct $Y \in N_X^1(\F)$ and $Y' \in N_{X'}^1(\F)$
such that $Y \cap Y' = C_S$, by Lemma~\ref{l:distMsInStars}.
By the definitions of  $N_X^1(\F)$ and $N_{X'}^1(\F)$, $M=X\ltimes_1 Y$ and $M'=X' \ltimes_1 Y'$ are both in $\F$ and must satisfy
\[
R_k(M \cap M')\subseteq  R_k(M)\cap R_k(M')=X\cap X' =C_R 
\]
and
\[
S_k^1(M \cap M')\subseteq  S_k^1(M)\cap S_k^1(M')=Y\cap Y' =C_S\,.
\]
As $\F$ is $t$-intersecting, $|M \cap M'|\geq t$ and we must have  that $|M\cap M'|=t$, $R_k(M \cap M')=C_R$
and $S_k^1(M \cap M')=C_S$, since $|C_R|=t=|C_S|$. It follows that 
\[
V_1(C_R) = V_1(R_k(M\cap M'))= V_1(M\cap M') =  V_1(S^1_k(M\cap M'))= V_1(C_S).
\]


Therefore $C_R$ and $C_S$ are compatible; so let $C=C_R\ltimes_1 C_S$.
As $C_R\subseteq R_k(M)$ and $C_S \subseteq S_k^1(M)$ for all $M \in \F$,
it follows that every  $M\in \F$ contains $C$.
So $\F$ is a subset of the $t$-star of $\K_r\n$ with centre $C$.
We complete the proof by showing that $\F$ is the $t$-star of $\K_r\n$ with centre $C$, by proving that $M \in \F$ for all $C \subseteq M \in \K_r\n$.
For such an $M$, $X=R_k(M) \supseteq R_k(C) = C_R$ and $S_k^1(M) \supseteq S_k^1(C)=C_S$.
As $R_k(\F)$ is the $t$-star of $R_k(\K_r\n)$ with centre $C_R$ and $N_X^1(\F)$ is the $t$-star of $K_r(V_1(X),[n_k])$ with centre $C_S$,
$R_k(M)\in R_k(\F)$ and $S_k^1(M) \in N_X^1(\F)$ from which it follows that 
$M = R_k(M) \ltimes_1 S_k^1(M)$ is a member of $\F$.
\end{proof}

\section{Main Results}\label{sec:MainResults}
In this section we prove the main results of this paper.
The proofs of Theorems~\ref{thm: Gen int perms} and~\ref{thm: gen int hyperperm non-uniform} are given shortly, while their $t$-intersecting analogues are proven at the end of the section.
For the latter results and Theorem~\ref{thm: gen int hyperperm non-uniform}, a short overview of previous related work is given for context.
We begin by proving Theorem~\ref{thm: Gen int perms}.
\begin{proof}[Proof of Theorem~\ref{thm: Gen int perms}]
Suppose without loss of generality that $n_1\leq \cdots \leq n_k$.
We proceed by induction on $k$.
The base case, $k=2$, is Theorem~\ref{thm: gen perm}.
So, suppose $k \geq 3$ and the theorem holds for $k-1$.
The result is trivial if $n_i\leq 2$ for all $i$.
So we may assume that $n_k \geq 3$.
By Lemma~\ref{lem:properties}(i) and (ii),
$N_X^1(\mathscr{F})$ is intersecting
and $N_X^1(\mathscr{F}) \subseteq K_{r}(V_1(X),[n_k])$ for all $X\in R_k(\F)$, respectively.
Since $|V_1(X)| = r$ and $N_X^1(\mathscr{F})$ is an intersecting subfamily of $K_{r}(V_1(X),[n_k])$,
Theorem~\ref{thm: gen perm} implies that for each $X\in R_k(\F)$
\begin{equation}\label{eqn: int thm FRij}
\big|N^1_X(\mathscr{F})\big| \leq
\frac{(|V_1(X)|-1)_{r-1} (n_k-1)_{r-1}}{(r-1)!}
= (n_k-1)_{r-1}\,,
\end{equation}
with equality if and only if $N_X^1(\F)$ is a star of $K_{r}(V_1(X),[n_k])$.
By Lemma~\ref{lem:properties}(i), 
$R_k(\mathscr{F}) \subseteq R_k(\K_r\n)$ is also intersecting.
Thus, by the inductive hypothesis,
\begin{equation}\label{eqn: int thm Rij}
|R_k(\mathscr{F})| \leq \frac{(n_1-1)_{r-1} \cdots (n_{k-1}-1)_{r-1}}{(r-1)!}\,,
\end{equation}
with equality if and only if $R_k(\F)$ is a star of $R_k(\K_r\n)$.
Then by Lemma~\ref{lem:properties}(iv) and inequalities (\ref{eqn: int thm FRij}) and (\ref{eqn: int thm Rij})
\begin{equation}\label{eqn:bound of thm t=1}
     |\mathscr{F}|
  =  \sum_{X \in R_k(\mathscr{F})} |N^1_X(\mathscr{F})|
\leq \sum_{X \in R_k(\mathscr{F})}(n_k-1)_{r-1}
\leq \frac{(n_1-1)_{r-1} \cdots (n_k-1)_{r-1}}{(r-1)!}\,,
\end{equation}
which establishes the bound of the theorem.

So it suffices to show that maximum sized families are stars of $\K_r\n$.
When $\F$ is a star of $\K_r\n$ it is easy to check that $\F$ attains the bound in the theorem.
So suppose that $\F$ attains the bound of the theorem and so all the inequalities in (\ref{eqn:bound of thm t=1}) are equalities.
Then the inequality in (\ref{eqn: int thm FRij}) for each $X \in R_k(\F)$ and the inequality in (\ref{eqn: int thm Rij}) are also equalities
from which it follows that
$N_X^1(\F)$ is a star of $K_r(V_1(X),[n_k])$ for all $X \in R_k(\F)$ 
and $R_k(\F)$ is a star of $R_k(\K_r\n)$.
Hence as $n_k \geq 3$, Lemma~\ref{lem:whenSubfamststar} implies that $\F$ must be a star of $\K_r\n $.
\end{proof}

Erd\H{o}s, Ko and Rado~\cite{MR0140419} also considered intersecting families $\F \subseteq 2^{[n]}$ and proved
such a family has size at most $2^{n-1}$.
A star of $2^{[n]}$, i.e., the collection of all subsets of $[n]$ containing a fixed element $x \in [n]$,
has the maximum size, but not uniquely so. A general description of all intersecting families $2^{[n]}$ with size $2^{n-1}$
has not been found for general $n$; see~\cite{MR3066347} for an enumeration of the number of intersecting subfamilies of $2^{[n]}$ for small $n$. Here we prove an analogous result for $k$-partite $k$-graphs which, contrary to the result of Erd\H{o}s et al.\ \cite{MR0140419}, has a relatively simple classification of maximum sized families.

Here and henceforth we let $R$ denote a set of integers $R \in 2^{[n_1]}$ such that $\min\{r\in R\} \geq t$.
Let $\K_R \n = \bigcup_{r \in R} \K_r\n$. 
It follows from the definition of a $t$-star that 
the $t$-star of $\K_R(n_1,\ldots,n_k)$ with centre $C \in \K_t\n$ is the family
$\{M\in \K_R\n\;:\; C \subseteq M\}$.
Equivalently, the $t$-star of $\K_R(n_1,\ldots,n_k)$ with centre $C$
is the union of the $t$-stars of $\K_r(n_1,\ldots,n_k)$ with centre $C$ over $r\in R$.
Of course, analogous properties hold for stars of $\K_R(n_1,\ldots,n_k)$.
We prove an analogue of Theorem~\ref{thm: Gen int perms}
for families $ \mathscr{F} \subseteq \K_R\n$.
\begin{thm}\label{thm: gen int hyperperm non-uniform}
Let $\mathscr{F} \subseteq \K_R\n$ be an intersecting family.
Then
\[
  |\mathscr{F}| \leq \sum_{r \in R} \frac{(n_1-1)_{r-1} \cdots (n_k-1)_{r-1}}{(r-1)!}\,.
\]
Furthermore, equality is attained if and only if $\mathscr{F}$ is a star of $\K_R\n$.
\end{thm}

\begin{proof}
If $n_i \leq 2$ for all $i \in [k]$, then $R \subseteq [2]$ and simple case analysis
shows that the result holds.
So suppose that $n_k\geq 3$.
Let $\mathscr{F}_{r} = \mathscr{F} \cap \K_r\n$ for all $r \in R$.
By Theorem~\ref{thm: Gen int perms},
\begin{equation}\label{eqn: F non unif t=1}
       |\mathscr{F}|
    =  \sum_{r \in R} |\mathscr{F}_{r}|
  \leq \sum_{r \in R} \frac{(n_1-1)_{r-1} \cdots (n_k-1)_{r-1}}{(r-1)!} \,,
\end{equation}
and equality holds if and only if
$\mathscr{F}_{r}$ is a star of $\K_r\n$ for all $r \in R$.
We complete the proof by showing that $\F$ has maximum size if and only if $\F$ is a star of $\K_R\n$.
Clearly a star of $\K_R\n$ has maximum size.
If $\F$ has maximum size, then the inequality in (\ref{eqn: F non unif t=1}) is an equality
and so $\mathscr{F}_{r}$ is a star of $\K_r\n$ for all $r\in R$.
As the union of the stars $\F_r$ over $r\in R$ is intersecting and $n_k\geq 3$,
Lemma~\ref{l:distMsInStars} implies that all the stars $\F_r$ must have the same centre $C$.
Therefore, $\F$ is the star of $\K_R\n$ with centre $C$.
\end{proof}

\subsection{$t$-intersecting theorems for $k$-partite $k$-graphs}\label{sec:proof of thm for t-intersecting}

In this subsection we prove $t$-intersecting analogues of Theorems~\ref{thm: Gen int perms} and~\ref{thm: gen int hyperperm non-uniform}.
First we briefly provide some background of previous work related to $t$-intersecting families of sets and permutations.
Erd\H{o}s et al.\ \cite{MR0140419} also proved that for $r>t$ and $n$ sufficiently larger than $r$ any $t$-intersecting subfamily $\F$ of $\binom{[n]}{r}$ has size at most $\binom{n-t}{r-t}$ and moreover $\F$ has exactly that size if and only if $\F$ is a $t$-star of $\binom{[n]}{r}$.
Frankl~\cite{MR519277} proved that result for $n \geq(t+1)(r-t+1)$ and $t\geq 15$. Finally, Wilson \cite{MR771733} proved the same result for all $n \geq (t+1)(r-t+1)$ and $t\geq 1$, except the uniqueness of maximum sized families when $n =(t+1)(r-t+1)$.
This is the best possible with respect to $n$ as larger families than $t$-stars exist for $n <(t+1)(r-t+1)$ and families equally as large as $t$-stars exist when $n =(t+1)(r-t+1)$; see the Conclusion.

Frankl and Deza~\cite{MR0439648} conjectured that maximum sized $t$-intersecting family of permutations are $t$-stars for $n$ sufficiently larger than $t$.
This conjecture was later confirmed by Ellis, Friedgut, and Pilpel~\cite{MR2784326}.
Meagher and Razafimahatratra~\cite{MeRa} proved that a $2$-intersecting subfamily of $\fP_n(n,n)$
has size at most $(n-2)!$ (the size of a $2$-star of $\fP_n(n,n)$) and gave a algebraic characterisation of maximum sized families for all $n\geq 2t+1$.
More recently, Keller et al.\ \cite{KLMS24} proved, using a method based on hypercontractivity for global functions, that there is a universal constant $c_0$ such that for all $t \leq c_0 n$ any $t$-intersecting family of permutations $\F \subseteq\fP_n(n,n)$ has size at most $(n-t)!$. Furthermore, they showed that if $|\F| \geq 0.75  (n-t)!$, then $\F$ is contained in a $t$-star.
Kupavskii and Zakharov~\cite{KuZa24}, using an approach they call spread approximations method, were able to prove a similar result, namely that any $t$-intersecting family of permutations $\F \subseteq\fP_n(n,n)$ has size at most $(n-t)!$ for any $t$ and $n$ which satisfy $n>2^{22}t\log_2(n)^2$ and, if $|\F| > \frac{2}{3}(n-t)!$, then $\F$ is contained in a $t$-star.
Crucially, Kupavskii and Zakharov~\cite{KuZa24} were are able to prove their result without exploiting the properties of $\fP_n(n,n)$, most notably the representation theory of $\fP_n(n,n)$, unlike many other proofs which heavily rely on the algebraic structure of $\fP_n(n,n)$.

Borg~\cite{MR2526749} proved the following result for $t$-intersecting families $\F \subseteq K_R(n_1,n_2)$ for $\max\{n_1,n_2\}$ sufficiently larger than $r$ and $t$.

\begin{thm}\label{thm:t-int gen perms non-uniform}
Let $\F \subseteq K_R(n_1,n_2)$, where
$n_2$ is sufficiently larger than $\max\{r\in R\}$ and $t$.
Then 
\[
|\mathscr{F}| \leq \sum_{r \in R}\frac{(n_1-t)_{r-t} (n_2-t)_{r-t}}{(r-t)!} \,.
\]
Furthermore, equality is attained if and only if $\mathscr{F}$ is a $t$-star of $K_R(n_1,n_2)$.
\end{thm}
\noindent
Brunk and Huczynska~\cite{MR2587034} independently proved Theorem~\ref{thm:t-int gen perms non-uniform} for the cases when $R = \{r\}$ and $r=n_1$.
An analogous result was found for families of perfect matchings of the complete graph that are 2-intersecting by Fallat, Meagher and Shirazi~\cite{FMS}.

Using Theorem~\ref{thm:t-int gen perms non-uniform} we will prove the following $t$-intersecting analogue of Theorem~\ref{thm: Gen int perms}
and subsequently a $t$-intersecting analogue of Theorem~\ref{thm: gen int hyperperm non-uniform}.
\begin{thm}\label{thm: Gen t-int perms hyper}
Let $\mathscr{F} \subseteq \K_r\n $ be a $t$-intersecting family, where $n_2$, the second smallest $n_i$, is sufficiently larger than $r$ and $t$. Then
\[
  |\mathscr{F}| \leq \frac{(n_1-t)_{r-t} \cdots (n_k-t)_{r-t}}{(r-t)!} \,.
\]
Furthermore, equality is attained if and only if $\mathscr{F}$ is a $t$-star of $\K_r\n$.
\end{thm}

\begin{proof}
We proceed by induction on $k$.
The base case, $k=2$, is Theorem~\ref{thm:t-int gen perms non-uniform} for $R = \{r\}$.
So, suppose $k \geq 3$ and the theorem holds for  $k-1$.
By Lemma~\ref{lem:properties}(i) and (ii),
$N_X^1(\mathscr{F})$ is $t$-intersecting
and $N_X^1(\mathscr{F}) \subseteq K_{r}(V_1(X),[n_k])$ for all $X\in R_k(\F)$.
Since $|V_1(X)| = r$, $N_X^1(\mathscr{F})$ is a $t$-intersecting subfamily of $K_{r}(V_1(X),[n_k])$
and $n_k$ is sufficiently larger than $r$ and $t$,
Theorem~\ref{thm:t-int gen perms non-uniform} (with $R=\{r\}$) implies that for each 
$X \in R_j(\F)$
\begin{equation}\label{eqn: t-int thm FRij}
\big|N^1_X(\mathscr{F})\big| \leq
\frac{(|V_1(X)|-t)_{r-t} (n_k-t)_{r-t}}{(r-t)!}
= (n_k-t)_{r-t}\,
\end{equation}
where equality is attained if and only if $N^1_X(\mathscr{F})$ is a $t$-star of $K_{r}(V_1(X),[n_k])$.
By Lemma~\ref{lem:properties}(i), 
$R_k(\mathscr{F}) \subseteq R_k(\K_r\n)$ is also $t$-intersecting.
As $n_2$ is sufficiently larger than $r$ and $t$,
by the inductive hypothesis,
\begin{equation}\label{eqn: t-int thm Rij}
|R_k(\mathscr{F})| \leq \frac{(n_1-t)_{r-t} \cdots (n_{k-1}-t)_{r-t}}{(r-t)!}\,,
\end{equation}
with equality if and only if $R_k(\mathscr{F})$ is a $t$-star of $R_k(\K_r\n)$.
Then by Lemma~\ref{lem:properties}(iv) and inequalities (\ref{eqn: t-int thm FRij}) and (\ref{eqn: t-int thm Rij})
\begin{equation}\label{eqn:bound of thm}
     |\mathscr{F}|
  =  \sum_{X \in R_k(\mathscr{F})} |N^1_X(\mathscr{F})|
\leq \sum_{X \in R_k(\mathscr{F})}(n_k-t)_{r-t}
\leq \frac{(n_1-t)_{r-t} \cdots (n_k-t)_{r-t}}{(r-t)!}\,,
\end{equation}
which establishes the bound of the theorem.

So it suffices to prove that maximum sized families are $t$-stars of $\K_r\n$.
When $\F$ is a $t$-star of $\K_r\n$ it is easy to check that $\F$ has maximum size.
So suppose that $\F$ has maximum size and so the inequalities in (\ref{eqn:bound of thm}) are equalities.
Then equality holds in (\ref{eqn: t-int thm FRij}) for each $X \in R_k(\F)$ and equality holds in (\ref{eqn: t-int thm Rij})
from which it follows that $N_X^1(\F)$ is a $t$-star of $K_r(V_1(X),[n_k])$ for all $X \in R_k(\F)$ and $R_k(\F)$ is a $t$-star of $R_k(\K_r\n)$.
Thus by Lemma~\ref{lem:whenSubfamststar}, $\F$ must be a $t$-star of $\K_r\n $.
\end{proof}
In fact, we expect that Theorem~\ref{thm: Gen t-int perms hyper} is true even if the condition 
``$n_2$ is sufficiently larger than $r$" is replaced with ``$n_k$ is sufficiently larger than $t$" (and $r \leq n_1$ is the only condition placed on the size of $r$ compared with each of $n_1,\ldots, n_k$). 
However we can say something about the special case when $r=n_1= \cdots=n_k$ as follows.
By using Ellis et al.\ \cite{MR2784326} result that any maximum sized $t$-intersecting family of permutations is a $t$-star for $n$ sufficiently larger than $t$ as the base case, the proof above can easily be amended to prove the following. 
\begin{thm}\label{thm:t-intHyperPerm}
Let $\F \subseteq \K_n(n,\ldots ,n)$ be a $t$-intersecting family, where $\K_n(n,\ldots ,n)$ has $k$ parts. Then for $n$ sufficiently larger than $t$
\[
  |\mathscr{F}| \leq \frac{((n-t)_{r-t})^{k}}{(r-t)!}.
\]
Furthermore, equality is attained if and only if $\F$ is a $t$-star of $\K_n(n,\ldots ,n)$.
\end{thm}
After the submission this paper, the results of Keller et al.\ \cite{KLMS24} and Kupavskii and Zakharov~\cite{KuZa24} on $t$-intersecting families of permutations describe earlier were published. Either of these results could be used in place of Ellis et al.\ \cite{MR2784326} result to strengthen the theorem above slightly, by requiring a weaker condition on $n$.

Katona~\cite{MR0168468} proved a $t$-intersecting analogue of Erd\H{o}s et al.\ \cite{MR0140419}
result for intersecting families of $2^{[n]}$.
Perhaps surprisingly, $t$-stars of $2^{[n]}$ are not maximum sized $t$-intersecting families
of $2^{[n]}$. We end the section by proving the following $t$-intersecting analogue of Theorem~\ref{thm: gen int hyperperm non-uniform}, which suggests that the behaviour of $r$-matchings of $\K\n$ differs from sets in that regard, in general.


\begin{thm}\label{thm:t-int gen perms hyper non-uniform}
Let $\mathscr{F} \subseteq \K_R\n$ be a $t$-intersecting family, where
$n_2$ is sufficiently larger than $\max\{r\in R\}$.
Then
\[
  |\mathscr{F}| \leq \sum_{r \in R} \frac{(n_1-t)_{r-t} \cdots (n_k-t)_{r-t}}{(r-t)!}\,.
\]
Furthermore, equality is attained if and only if $\mathscr{F}$ is a $t$-star of $\K_R\n$.
\end{thm}

\begin{proof}
Let $\mathscr{F}_{r} = \mathscr{F} \cap \K_r\n$ for all $r \in R$.
By Theorem~\ref{thm: Gen t-int perms hyper},
\begin{equation}\label{eqn: F non unif}
       |\mathscr{F}|
    =  \sum_{r \in R} |\mathscr{F}_{r}|
  \leq \sum_{r \in R} \frac{(n_1-1)_{r-1} \cdots (n_k-1)_{r-1}}{(r-1)!} \,,
\end{equation}
and equality holds if and only if
$\mathscr{F}_{r}$ is a $t$-star of $\K_r\n$, for all $r \in R$.
We complete the proof by showing that $\F$ has maximum size if and only if $\F$ is a $t$-star of $\K_R\n$.
Clearly a $t$-star of $\K_R\n$ has maximum size. 
If $\F$ has maximum size, then the inequality in (\ref{eqn: F non unif}) is an equality
and so $\mathscr{F}_{r}$ is a $t$-star of $\K_r\n$ for all $r\in R$.
By Lemma~\ref{l:distMsInStars}, the union of the $t$-stars $\F_r$ over $r\in R$ is $t$-intersecting only if each $\F_r$ has the same centre. 
Therefore, each $\F_r$ must have the same centre $C$ from which it follows that $\F$ is the $t$-star of $\K_R\n$ with centre $C$.
\end{proof}
As with Theorem~\ref{thm: Gen t-int perms hyper}, we expect the above theorem to hold under the weaker assumption that $n_k$ is sufficiently larger than $t$ and no conditions (other than the obvious condition $\max\{r\in R\}\leq n_1$) are placed on the size of the elements of $R$ with respect to $n_1,\ldots , n_k$.

\section{Concluding Remarks}\label{sec:conclusion}
%

The following families demonstrate that there are $t$-intersecting subfamilies of $\binom{[n]}{r}$
which are larger than $t$-stars of $\binom{[n]}{r}$, whenever $n < (t+1)(r-t+1)$.
For each $0 \leq i\leq r-t$, and $C \in \binom{[n]}{t+2i}$ let
\[
\G_i(C) = \left\lbrace G\in \binom{[n]}{r}\;:\; |G \cap C| \geq t+i\right\rbrace\,.
\]
Ahlswede and Khachatrian~\cite{MR1429238} proved that any maximum sized $t$-intersecting subfamily of 
$\binom{[n]}{r}$ is $\G_i(C)$ for some $i$ and $C$. In very special cases of $n$, $\G_i(C)$
and $\G_{i+1}(C')$ can be the same size and both be maximum sized $t$-intersecting families. In particular, when $n =(t+1)(r-t+1)$ the families $\G_0(C)$ and $\G_1(C')$ both are maximum sized $t$-intersecting families; as the former is in fact a $t$-star of $\binom{[n]}{r}$, this explains why
the uniqueness of maximum sized families could not be obtained when $n =(t+1)(r-t+1)$ in Wilson's~\cite{MR771733} result mentioned earlier.


Natural analogues of the families $\G_i(C)$  have been proposed previously for permutations and $r$-matchings of $K(n_1,n_2)$; see for example \cite{MR2784326} and~\cite{MR2587034}. Here we give the more general natural analogue of $\G_i(C)$ for $r$-matchings of $\K\n$,
as we will discuss these shortly. Let $0 \leq i\leq \min\{r-t,\frac{n_1-t}{2}\}$, and  $C \in \K_{t+2i}\n$. Then we define: 
\[
\fH_i(C) =\{M \in \K_r\n\;:\; |M \cap C| \geq t+i\}\,.
\]
As with the families $\G_i(C)$,
every $\fH_i(C) $ is $t$-intersecting and in special cases can be larger than $t$-stars, as the following example demonstrates.
\begin{exa}
Let $k=2$, $t=4$, $r=n_1=n_2=8$ and $C = \{(j,j): j \in [6]\}$. Then 
$\fH_1(C)$ has size $(3!-2!)\binom{6}{5} + 2!=26$, while a $4$-star of $K_8(8,8)$ has size $4!=24$.
\end{exa}
Ellis~et al.\ \cite{MR2784326} conjectured that any maximum $t$-intersecting family of permutations must be 
$\fH_i(C)$ for some $i$ and $C$. This conjecture remains largely open for cases when $t$-stars are not maximum sized $t$-intersecting families. Some progress has been made recently by Ellis and Lifshitz ~\cite{ElLi22} who established the conjecture for $t = O(\frac{\ln n}{\ln(\ln n)})$.
Brunk and Hucvzynska~\cite{MR2587034} made a similar yet more precise conjecture to that of Ellis~et al.\ \cite{MR2784326} for $t$-intersecting subfamilies of $K_r(r,n_2)$, by also specifying which of the families $\fH_i(C)$ should have maximum size for all values of $n_2$.
Brunk and Hucvzynska (see Theorem 3.5~\cite{MR2587034}) were able to prove that an $(r-a)$-intersecting subfamily of $K_{r}(r,r+b)$ must be $\fH_i(C)$ for some $i$ and $C$, provided that $a$ and $b$ are positive and $r$
is sufficiently larger than $a$ and $b$.

Naturally we expect a similar result to the conjectures of Ellis~et al.\ \cite{MR2784326} and Brunk and Hucvzynska~\cite{MR2587034} should hold for $t$-intersecting families $ \F \subseteq \K_r\n$.
However we now demonstrate why one cannot expect such a result to follow from the conjectures of Ellis~et al.\ \cite{MR2784326} and Brunk and Hucvzynska~\cite{MR2587034}
directly using the techniques of this paper.
Let $B_{r,t}(n_1,\ldots, n_k)$ be the maximum size of a $t$-intersecting family $ \F \subseteq \K_r\n$.
Using Lemma~\ref{lem:properties}(iv) and the definitions of $B_{r,t}(n_1,\ldots, n_{k-1})$ and $B_{r,t}(n_1, n_k)$, one can immediately obtain the following bound for the size of a $t$-intersecting family $ \F \subseteq \K_r\n$
\begin{equation}\label{e:gen_upper_bound}
|\mathscr{F}| = \sum_{X \in R_k(\mathscr{F})} |N_X^1(\mathscr{F})|\leq B_{r,t}(n_1,\ldots, n_{k-1})B_{r,t}(r, n_k)\,
\end{equation}
since $N_X^1(\mathscr{F})\subseteq K(V_1(X),[n_k])$, by Lemma~\ref{lem:properties}(ii), which is isomorphic to $K(r,n_k)$ for all $X \in R_k(\F)$.
To prove the bounds on $t$-intersecting families $\F\subseteq \K_r\n$ in Theorems~\ref{thm: Gen t-int perms hyper} and \ref{thm:t-intHyperPerm} it was sufficient to use \eqref{e:gen_upper_bound} as the bound in \eqref{e:gen_upper_bound} was tight for some families $\F$, namely $t$-stars of $\K_r\n$
and, for such families, $R_j(\F)$ and $N_X^i(\F)$ are maximum sized $t$-intersecting families for all $X \in R_k(\F)$.
%

We will prove that the bound in \eqref{e:gen_upper_bound} cannot be tight if 
$\fH_i(C)$ is a maximum sized $t$-intersecting family that is not a $t$-star. 
For simplicity we maintain the convention $r\leq n_1 \leq \cdots \leq n_k$.
Let $1 \leq i\leq \min\{r-t, \frac{n_1-t}{2}\}$  and  $C_{j} = \{(a,\ldots,a): a \in [t+i+j]\} \in \K_{t+i+j}\n$ for all $0 \leq j \leq i$ and consider $\fH_i(C_{i})$. 
Assume that $t+2<n_k$ when $i=1$ and $r=t+2=n_1$, so that $\fH_i(C_{i})$ does not trivially contain one element in that case. 
Let $X \in R_k(\fH_i(C_{i}))$ and let $D$ be the subset of $C_{i}$ such that $X \cap R_k(C_{i})=R_k(D)$ and $|D|=t+i+j$ where $0 \leq j \leq i$.
For any $Y \in N_X^1(\fH_i(C_{i}))$, $X \ltimes_1 Y$ must contain at least $t+i$ edges of $C_{i}$, by the definition of $\fH_i(C_{i})$ and, as $X \cap R_k(C_{i})=R_k(D)$, $(X \ltimes_1 Y) \cap C_i \subseteq D$. 
So in fact $(X \ltimes_i Y)$ contains at least $t+i$ edges of $D$ for all $Y \in N_X^1(\fH_i(C_{i}))$ and so every $Y \in N_X^1(\fH_i(C_{i}))$ contains at least $t+i$ edges of $S_k^1(D)$. 
Conversely, any $Y \in K(V_1(X),[n_k])$ that contains at least $t+i$ edges of $S_k^1(D)$ will be a member of $N_X^1(\fH_i(C_{i}))$, since $ |(X \ltimes_1 Y) \cap D| \geq t+i$ and thus $X \ltimes_1 Y \in \fH_i(C_{i})$ for such a $Y$.
Hence 
\[
N_X^1(\fH_i(C_{i}))=\left\{Y \in K_r(V_1(X),[n_k])\;:\; |Y \cap S_k^1(D)| \geq t+i \right\}\,.
\]
As the size of the above only depends on $j=|D|$ (and $r=|V_1(X)|$), 
for simplicity we only consider
\[
N_{j} = \left\{Y \in K_r(r,n_k)\;:\; |Y \cap S_k^1(C_{j})| \geq t+i \right\}
\]
for $0 \leq j \leq i$. Clearly, $N_{j} \subseteq N_{j+1}$ for $j \leq i-1$, since $|Y \cap S_k^1(C_{j})| \geq t+i $ implies that $|Y \cap S_k^1(C_{j+1})| \geq t+i$.
One can easily construct a matching $Y' \in K_r(r,n_k)$ 
which contains $(a,a) \in K_1(r,n_k)$ for all $j+2\leq a\leq t+i+j+1$ but no other elements of $S_k^1(C_{j+1})$, where, in the exceptional case when $i=1$, $j=0$ and $r=t+2=n_1$,
the assumption that $n_k>t+2$ ensures that there is an edge in $K_{t+2}(t+2,n_k)$ that is disjoint from all the edges in $S_k^1(C_{1})\setminus \{(1,1)\}$ that isn't $(1,1)$.
Clearly $|Y' \cap S_k^1(C_{j+1})|=t+i$ and $|Y' \cap S_k^1(C_{j})| = t+i-1$.
Thus $Y' \in N_{j+1}$ but $Y' \notin N_{j}$. 
It follows that the size of $N_j$ is increasing with respect to $j$ and, in particular, $N_0 ,\ldots , N_i$ can not all be the same size and so can not all be of the maximum size of a $t$-intersecting family. Thus the bound in \eqref{e:gen_upper_bound} is not tight.


\begin{thebibliography}{99}
\bibitem{MR1429238}
R.~Ahlswede and L.~H.~Khachatrian, 
The complete intersection theorem for systems of finite sets,
{\em European J. Combin.} {\bf 18(2)} (1997), 125--136.

\bibitem{BeZh}
A.~Berger and Y.~Zhao,
$K_4$-intersecting families of graphs,
{\em J. Combin. Theory Ser. B.} {\bf 163} (2023), 112--132.

\bibitem{MR2526749}
P.~Borg, 
On {$t$}-intersecting families of signed sets and permutations,
{\em Discrete Math.} {\bf 309(10)}  (2009), 3310--3317.

\bibitem{MR3301134}
P.~Borg and K.~Meagher, 
Intersecting generalised permutations,
{\em Australas. J. Combin.} {\bf 61}  (2015), 147--155.

\bibitem{MR3066347}
A.~E.~Brouwer, C.~F.~Mills, W.~H.~Mills and A.~Verbeek,
Counting families of mutually intersecting sets,
{\em Electron. J. Combin.} {\bf 20(2)} (2013), \#P8.

\bibitem{MR2587034}
F.~Brunk and S.~Huczynska, 
Some {Erd\H os-Ko-Rado theorems} for injections,
{\em European J. Combin.} {\bf 31(3)} (2010), 839--860.

\bibitem{MR2009400}
P.~J.~Cameron and C.~Y.~Ku,  
Intersecting families of permutations,
{\em European J. Combin.} {\bf 24(7)} (2003), 881--890.


%



\bibitem{EFF}
D.~Ellis, Y.~Filmus and E.~Friedgut,
Triangle-intersecting families of graphs,
{\em J. Eur. Math. Soc. (JEMS)} {\bf 14(3)} (2012), 841--885.

\bibitem{MR2784326}
D.~Ellis, E.~Friedgut and H.~Pilpel, 
Intersecting families of permutations,
{\em J. Amer. Math. Soc.}  {\bf 24(3)} (2011), 649--682.

\bibitem{ElLi22}
D.~Ellis and N.~Lifshitz,
Approximation by juntas in the symmetric group, and forbidden intersection problems,
{\em Duke Math. J.} {\bf 171(7)} (2022), 1417--1467.

\bibitem{MR0140419}
P.~Erd\H{o}s, C.~Ko and R.~Rado, 
Intersection theorems for systems of finite sets,
{\em Quart. J. Math. Oxford Ser.}(2) {\bf 12} (1961), 313--320.

\bibitem{FMS}
S.~ Fallat, K.~Meagher and M.~Shirazi,
The Erdős–Ko–Rado theorem for 2-intersecting families of perfect matchings,
{\em Algebr. Comb.} {\bf 4(4)} (2021), 575--598.


\bibitem{MR519277}
P.~Frankl,
The Erd\H{o}s-Ko-Rado theorem is true for {$n=ckt$},
{\em Coll. Math. Soc. J. Bolyai.} {\bf 18} (1978), 365--375.

\bibitem{MR0439648}
P.~Frankl and M.~Deza, 
On the maximum number of permutations with given maximal or minimal
  distance,
{\em J. Combin. Theory Ser. A.} {\bf 22(3)} (1977), 352--360.

\bibitem{MR0168468}
G.~Katona, 
Intersection theorems for systems of finite sets,
{\em Acta. Math. Acad. Sci. Hungar.} {\bf 15} (1964), 329--337.

\bibitem{KLMS24}
N.~Keller, N.~Lifshitz, D.~Minzer and O.~Sheinfeld,
On $t$-intersecting families of permutations,
{\em Adv. in Math.} {\bf 445}, (2024), Paper No. 109650.



\bibitem{MR2202076}
C.~Y.~Ku and I.~Leader, 
An {Erd\H{o}s}-{K}o-{R}ado theorem for partial permutations,
{\em Discrete Math.} {\bf 306(1)} (2006), 74--86.

\bibitem{KuZa24}
A.~Kupavskii and D.~Zakharov,
Spread approximations for forbidden intersections problems,
{\em Adv. Math.} {\bf 445} (2024), Paper No. 109653.


\bibitem{MR2061391}
B.~Larose and C.~Malvenuto, 
Stable sets of maximal size in {K}neser-type graphs,
{\em European J. Combin.} {\bf 25(5)} (2004), 657--673.

\bibitem{MR2285800}
Y.~S.~Li and J.~Wang, 
{Erd\H{o}s}-{K}o-{R}ado-type theorems for colored sets,
{\em Electron. J. Combin.} {\bf 14(1)} (2007), \#P1.9.

\bibitem{MeMo}
K.~Meagher and L.~Moura,
{Erd\H{o}s}-{K}o-{R}ado theorems for uniform set-partition systems,
{\em Electron. J. Combin.} {\bf 12} (2005), R40.

\bibitem{MeRa}
K.~Meagher and S.~A.~Razafimahatratra,
The Erdős-Ko-Rado theorem for 2-pointwise and 2-setwise intersecting permutations, {\em Electron. J. Combin.} {\bf 28(4)} (2021), \#P4.10, 21 pp.

\bibitem{SiSo76}
M.~Simonovits and V.~Sós,
Intersection theorems for graphs,
Probl\`{e}mes combinatories et th\'{e}orie des graphes.
{\em Colloq. Internat. CNRS.} {\bf 260} (1976), 389--391.

\bibitem{SiSo78}
M.~Simonovits and V.~Sós,
Intersection theorems for graphs II,
{\em Coll. Math. Soc. J\'{a}nos Bolyai} {\bf 18} (1978), 1017--1030.

\bibitem{MR771733}
R.~M.~Wilson, 
The exact bound in the {Erd\H{o}s}-{K}o-{R}ado theorem,
{\em Combinatorica} {\bf 4(2-3)} (1984), 247--257.







%
\end{thebibliography}
\end{document}